\documentclass{amsart}
\usepackage[a4paper,left=2.5cm, right=2.5cm, top=2.5cm, bottom=2.5cm]{geometry}
\usepackage{amssymb,hyperref}

\numberwithin{equation}{section}

\newcommand{\abs}[1]{\left\vert#1\right\vert}

\newcommand{\bb}[1]{\mathbb{#1}}

\usepackage{mathtools}

\newcommand*{\E}{\mathbb{E}}

\theoremstyle{plain}
\newtheorem{mythm}{Theorem}[section]
\newtheorem{theorem}{Theorem}[section]
\newtheorem{lemma}[mythm]{Lemma}
\newtheorem{proposition}[mythm]{Proposition}

\theoremstyle{definition}

\theoremstyle{remark}
\newtheorem{remark}{Remark}

\begin{document}
\begin{abstract}
Consider an infinite sequence $(U_n)_{n\in\bb{N}}$ of independent Cauchy random variables, defined by a sequence $(\delta_n)_{n\in\bb{N}}$ of location parameters and a sequence $(\gamma_n)_{n\in\bb{N}}$ of scale parameters. Let $(W_n)_{n\in\bb{N}}$ be another infinite sequence of independent Cauchy random variables defined by the same sequence of location parameters and the sequence $(\sigma_n\gamma_n)_{n\in\bb{N}}$ of scale parameters, with $\sigma_n\neq 0$ for all $n\in\bb{N}$. Using a result of Kakutani on equivalence of countably infinite product measures, we show that the laws of $(U_n)_{n\in\bb{N}}$ and $(W_n)_{n\in\bb{N}}$ are equivalent if and only if the sequence $(\abs{\sigma_n}-1)_{n\in\bb{N}}$ is square-summable.
\end{abstract}

\title[Quasi-invariance of countable products of Cauchy measures]{Quasi-invariance of countable products of Cauchy measures under non-unitary dilations}

\author{Han Cheng Lie}
\address{Institut f\"{u}r Mathematik, Freie Universit\"{a}t Berlin, Arnimallee 6, 14195 Berlin, Germany.}
\curraddr{Institut f\"{u}r Mathematik, Freie Universit\"{a}t Berlin, Arnimallee 6, 14195 Berlin, Germany.}
\email{hlie@math.fu-berlin.de}

\author{T.~J.~Sullivan}
\address{Free University of Berlin and Zuse Institute Berlin, Takustrasse 7, 14195 Berlin, Germany}
\curraddr{Free University of Berlin and Zuse Institute Berlin, Takustrasse 7, 14195 Berlin, Germany}
\email{sullivan@zib.de}



\date{\today}

\maketitle

\section{Introduction}
\label{sec:introduction}

Let $(U_n)_{n \in\bb{N}}$ be a sequence of independent, Cauchy random variables defined on a probability space $(\Omega,\mathcal{F},\bb{P})$, where each random variable $U_n$ has the density
\begin{equation}
	\label{eq:cauchy_density}
	f_n(x,\delta_n,\gamma_n) = \frac{1}{\pi \gamma_n} \frac{\gamma_n^2}{(x-\delta_n)^2+\gamma_n^2}.
\end{equation}
The distribution of $U_n$ is parametrised by the location and scale parameters $\delta_n$ and $\gamma_n$ respectively. We shall assume throughout that $\gamma_n$ is strictly positive for all $n\in\bb{N}$. For every $U_n$, we may define another Cauchy random variable $W_n$ by multiplicatively perturbing the scale parameter $\gamma_n$ by $\sigma_n\neq 0$, so that the pair $(\delta_n)_{n\in\bb{N}}$ and $(\sigma_n \gamma_n)_{n\in\bb{N}}$ determines the law of $(W_n)_{n\in\bb{N}}$. In this note, we prove the following result:
\begin{theorem}\label{thm:equivalence_laws_multiplicative_case}
Let $(U_n)_{n\in\bb{N}}$ be a sequence of independent Cauchy random variables defined by the sequences $(\delta_n)_{n\in\bb{N}}$ and $(\gamma_n)_{n\in\bb{N}}$ of location and scale parameters, and let $(W_n)_{n\in\bb{N}}$ be a sequence of independent Cauchy random variables with the sequences $(\delta_n)_{n\in\bb{N}}$ and $(\sigma_n\gamma_n)_{n\in\bb{N}}$ of location and scale parameters, where $\sigma_n\neq 0$ for all $n\in\bb{N}$. Then the laws of $(U_n)_{n\in\bb{N}}$ and $(W_n)_{n\in\bb{N}}$ are equivalent if and only if  $(\abs{\sigma_n}-1)_{n\in\bb{N}}$ is square-summable.
\end{theorem}
Equivalence of laws under additive perturbations to the sequence of locations or means has been studied for Gaussian $U_n$ \cite[Example 2.7.6]{bogachev_gaussian}, for sequences of random variables with finite Fisher information \cite{shepp}, and for stable Radon probability distributions on locally convex spaces \cite[Theorem 5.2.1]{bogachev_dmmc}. In contrast, Theorem \ref{thm:equivalence_laws_multiplicative_case} on multiplicative perturbations of scale parameters appears to be new. 

\section{Proof of Theorem \ref{thm:equivalence_laws_multiplicative_case}}
\label{sec:multiplicative}

We shall use Kakutani's theorem \cite[Theorem 1]{kakutani_equivalence_product_measures}, which we specialise to our context.
\begin{theorem}
	\label{thm:kakutani_equivalence_product_measures}
	Let $(U_n)_{n\in\bb{N}}$ and $(W_n)_{n\in\bb{N}}$ be given as in Theorem \ref{thm:equivalence_laws_multiplicative_case}. Then the laws of $(U_n)_{n\in\bb{N}}$ and $(W_n)_{n\in\bb{N}}$ are either equivalent or singular. The laws of $(U_n)_{n\in\bb{N}}$ and $(W_n)_{n\in\bb{N}}$ are equivalent if and only if both the following conditions hold: the laws $\bb{P}\circ U_n^{-1}$ and $\bb{P}\circ W_n^{-1}$ of $U_n$ and $W_n$ respectively are equivalent for every $n\in\bb{N}$, and
	\begin{equation}
		\label{eq:kakutani_equivalence_product_measures_suffcon}
		\sum_{n\in\bb{N}}-\log \E \left[\sqrt{\varphi_n(U_n,\abs{\sigma_n})}\right]<\infty
	\end{equation}
	where $\varphi_n(~\cdot~,\abs{\sigma_n})$ is the Radon--Nikodym derivative of $\bb{P}\circ W_n^{-1}$ with respect to $\bb{P}\circ U_n^{-1}$.
\end{theorem}
If $\sigma_n\neq 0$, then the laws of $U_{n}$ and $W_{n}$ are equivalent for every $n\in\bb{N}$, with 
\begin{equation}\label{eq:radon_nikodym_derivative_multiplicative_perturbation_scale}
 \varphi_n(x,\abs{\sigma_n})=\abs{\sigma_n}\frac{(x-\delta_n)^2+\gamma_n^2}{(x-\delta_n)^2+\sigma_n^2\gamma_n^2}.
\end{equation}
Let $y_n\coloneqq (x-\delta_n)/\gamma_n$. From \eqref{eq:cauchy_density} and \eqref{eq:radon_nikodym_derivative_multiplicative_perturbation_scale} it follows that
\begin{equation}\label{eq:integral_form_EVal_sqrt_RNdensity_mult_case}
 \bb{E}[\sqrt{\varphi_n(U_n,\abs{\sigma_n})}]=\frac{\sqrt{\abs{\sigma_n}}}{\pi}\int_{\bb{R}}\frac{1}{\sqrt{y_n^2+\sigma_n^2}\sqrt{y_n^2+1}}~\mathrm{d}y_n=\frac{1}{\pi}\int_{\bb{R}}\frac{1}{\sqrt{\tfrac{y_n^2}{\abs{\sigma_n}}+\abs{\sigma_n}}\sqrt{y_n^2+1}}~\mathrm{d}y_n.
\end{equation}
We use \eqref{eq:integral_form_EVal_sqrt_RNdensity_mult_case} in the following lemma.
\begin{lemma}\label{lem:necessary_condition_of_convergence_of_kakutanis_series}
 For all $\abs{\sigma_n}>0$, $-\log\bb{E}[\sqrt{\varphi_n(U_n,\abs{\sigma_n})}]$ is nonnegative, and is strictly positive if and only if $\abs{\sigma_n}\neq 1$. Furthermore, if \eqref{eq:kakutani_equivalence_product_measures_suffcon} holds, then $\lim_{n\to\infty}\abs{\sigma_{n}}=1$. 
\end{lemma}
\begin{proof}
Using Jensen's inequality with the concave map $x\mapsto \sqrt{x}$, and using that $\varphi_n$ is the Radon--Nikodym derivative of $\bb{P}\circ W_n^{-1}$ with respect to $\bb{P}\circ U_n^{-1}$, we obtain
\begin{equation*}
 1=\sqrt{\bb{E}[\varphi_n(U_n,\abs{\sigma_{n}})]}\geq \bb{E}[\sqrt{\varphi_n(U_n,\abs{\sigma_{n}})}],
\end{equation*}
and hence $ -\log \bb{E}[\sqrt{\varphi_n(U_n,\abs{\sigma_{n}})}]$ is nonnegative for all $\abs{\sigma_n}>0$. Since $x\mapsto \sqrt{x}$ is not affine, equality holds above (and hence $-\log\bb{E}[\sqrt{\varphi_n(U_n,\abs{\sigma_n})}]=0$) if and only if $\varphi_n(U_n,\abs{\sigma_n})$ is $\bb{P}$-almost surely constant. By \eqref{eq:integral_form_EVal_sqrt_RNdensity_mult_case}, $\varphi_n(U_n,\abs{\sigma_n})$ is $\bb{P}$-almost surely constant if and only if $\abs{\sigma_n}=1$. This proves the first statement. 

To prove the second statement, observe that  \eqref{eq:integral_form_EVal_sqrt_RNdensity_mult_case} implies continuity of the map $\abs{\sigma_n}\mapsto \bb{E}[\sqrt{\varphi_n(U_n,\abs{\sigma_n})}]$ on $(0,\infty)$. The dominated convergence theorem and the second equality in \eqref{eq:integral_form_EVal_sqrt_RNdensity_mult_case} imply that, if $(\abs{\sigma_n})_{n\in\bb{N}}$ decreases to zero or increases to infinity, then $\lim_{n\to\infty}\bb{E}[\sqrt{\varphi_n(U_n,\abs{\sigma_n})}]=0$. Summarising, we have that the map $\abs{\sigma_n}\mapsto \bb{E}[\sqrt{\varphi_n(U_n,\abs{\sigma_n})}]$ from $(0,\infty)$ to $[0,1]$ is continuous, attains the value 1 if and only if its argument is 1, and does not increase asymptotically to 1 at either end of its domain. It follows that $\lim_{n\to\infty}\bb{E}[\sqrt{\varphi_n(U_n,\abs{\sigma_n})}]=1$ if and only if $\lim_{n\to\infty}\abs{\sigma_n}=1$. Since \eqref{eq:kakutani_equivalence_product_measures_suffcon} implies that $\lim_{n\to\infty}\bb{E}[\sqrt{\varphi_n(U_n,\abs{\sigma_n})}]=1$, the proof is complete.
\end{proof}

\begin{remark}\label{rem:multiplicative_perturbations_nonzero_etc}
A consequence of Lemma \ref{lem:necessary_condition_of_convergence_of_kakutanis_series} that we shall use repeatedly is the following: if \eqref{eq:kakutani_equivalence_product_measures_suffcon} holds, and if  $\abs{\sigma_n}\neq 1$ for all $n\in\bb{N}$, then, for any $0<\alpha<1$, there exists some $N(\alpha)\in\bb{N}$ such that 
 \begin{equation}\label{eq:multiplicative_perturbations_nonzero_etc}
  0<\abs{\abs{\sigma_n}-1}<\alpha,\quad\forall n\geq N(\alpha).
 \end{equation}
 To simplify the subsequent presentation, we shall assume that $\alpha$ is sufficiently small for our purposes, and that $N(\alpha)=1$. If \eqref{eq:kakutani_equivalence_product_measures_suffcon} holds, then by Lemma \ref{lem:necessary_condition_of_convergence_of_kakutanis_series} we may assume this without any loss of generality, since we may discard finitely many terms in the series \eqref{eq:kakutani_equivalence_product_measures_suffcon} without losing convergence. 
\end{remark}

Given \eqref{eq:integral_form_EVal_sqrt_RNdensity_mult_case}, we may define for $0<\alpha<1$, a function $I:(1-\alpha,1+\alpha)\to (0,\infty)$ by
\begin{equation}\label{eq:I_of_abs_sigma}
 I(\abs{\sigma})\coloneqq \frac{1}{\pi}\int_{\bb{R}}\frac{1}{\sqrt{y^2+\sigma^2}}\frac{1}{\sqrt{y^2+1}}\ \mathrm{d}y.
\end{equation}
Therefore, $I(\abs{\sigma})=1$ if and only if $\abs{\sigma}=1$. Furthermore, for some $0<\alpha<1$ (see Remark \ref{rem:multiplicative_perturbations_nonzero_etc}), it follows from \eqref{eq:multiplicative_perturbations_nonzero_etc} that the integrand on the right-hand side of \eqref{eq:I_of_abs_sigma} is dominated by an integrable function that depends on $\alpha$ but not on $\sigma$. Thus, by the dominated convergence theorem, we may interchange integration and differentiation to compute the derivative of $I$ with respect to $\abs{\sigma}$. By computing the higher-order derivatives of $(y^2+\sigma^2)^{-1/2}$ with respect to $\abs{\sigma}$ and employing the same argument, the assertion holds for higher-order derivatives of $I$ as well. The Taylor expansion of $I(\abs{\sigma})$ about $\abs{\sigma}=1$ up to second order is
\begin{subequations}
 \begin{align}
  I(\abs{\sigma})&=1+a_1 (\abs{\sigma}-1)+a_2(\abs{\sigma}-1)^2+O((\abs{\sigma}-1)^3),
  \label{eq:series_representation_of_I_n_pre}
  \\
  a_m&=\frac{1}{m !}\frac{1}{\pi}\int_{\bb{R}}\frac{1}{\sqrt{y^2+1}}
  \left. \left(\frac{\partial}{\partial \abs{\sigma}}\right)^m\frac{1}{\sqrt{y^2+\abs{\sigma}^2}}  \right\vert_{\abs{\sigma}=1}
  \mathrm{d}y.
  \label{eq:a_ell_as_ellth_derivative_of_I_n}
 \end{align}
\end{subequations}
Note that the $a_m$ are independent of $\abs{\sigma}$. 

We shall use the following lemma to calculate the values of $a_m$ for $m=1,2$.
\begin{lemma}\label{lem:gamma_function_stuff}
 Let $r,s\in\bb{N}_0$ with $r\le s-1$. Then
 \begin{equation*}
  \int_{\bb{R}}\frac{x^{2r}}{(x^2+1)^s}\mathrm{d}x=\pi\frac{(2r)!(2(s-r-1))!}{4^{s-1} r!(s-r-1)!(s-1)!}
 \end{equation*}
\end{lemma}
\begin{proof}
 Letting $y\coloneqq x^2$, we have $\mathrm{d}x=2y^{-1/2}\mathrm{d}y$, so
 \begin{equation*}
  \int_{\bb{R}}\frac{x^{2r}}{(x^2+1)^s}\mathrm{d}x=\int_0^\infty\frac{y^{r-1/2}}{(y+1)^s}\mathrm{d}y=\int_0^\infty\frac{y^{(r+1/2)-1}}{(y+1)^s}\mathrm{d}y.
 \end{equation*}
Using \cite[Equation (2)]{tricomi_erdelyi} and properties of the gamma function $\Gamma(t)=\int_0^\infty x^{t-1}e^{-x}\mathrm{d}x$ (for $\text{Re } t>0$), we have, for $0<\text{Re } (r+1/2)< \text{Re } s$,
\begin{equation*}
 \int_0^\infty\frac{y^{(r+1/2)-1}}{(y+1)^s}\mathrm{d}y=\frac{\Gamma\left(r+\frac{1}{2}\right)\Gamma\left(s-r-\frac{1}{2}\right)}{\Gamma(s)}=\frac{(2r)!}{4^r r!}\sqrt{\pi}\frac{(2(s-r-1))!}{4^{s-r-1}(s-r-1)!}\sqrt{\pi}\frac{1}{(s-1)!}.
\end{equation*}
Simplifying the right-hand side yields the desired conclusion.
\end{proof}
Thus we have
 \begin{subequations}\label{eq:a}
\begin{align}
 a_1&=\frac{1}{\pi}\int_{\bb{R}}\frac{1}{\sqrt{y_n^2+1}}\frac{(-1)}{(y_n^2+1)^{3/2}}\mathrm{d}y_n=-\frac{1}{\pi}\left(\pi\frac{0!\ 2!}{4\cdot 1!\cdot 1!}\right)=-\frac{1}{2},
 \label{eq:a_1}
 \\
 a_2&=\frac{1}{2\pi}\int_{\bb{R}}\frac{1}{\sqrt{y_n^2+1}}\frac{2-y_n^2}{(y_n^2+1)^{5/2}}\mathrm{d}y_n=\frac{1}{2\pi}\left(\frac{9\pi}{8}-\frac{\pi}{2}\right)=\frac{5}{16}.
 \label{eq:a_2}
\end{align}  
\end{subequations}

Given the function $I$, a sequence $(\sigma_n)_{n\in\bb{N}}$ satisfying \eqref{eq:multiplicative_perturbations_nonzero_etc}, and \eqref{eq:a_2}, define 
\begin{equation}\label{eq:epsilon_n_mult_case}
 \epsilon_n\coloneqq \frac{1}{\abs{\sigma_n}-1}\left(I(\abs{\sigma_n})-1-a_1(\abs{\sigma_n}-1)\right),\quad\forall n\in\bb{N}.
\end{equation}
\begin{lemma}\label{lem:lower_upper_bounds_for_abs_epsilon_n}
Let $\epsilon_{n}$ be as in \eqref{eq:epsilon_n_mult_case}. If \eqref{eq:kakutani_equivalence_product_measures_suffcon} holds, then for some $0<\alpha<1$ and some $C>0$ that do not depend on $n$, 
 \begin{equation}\label{eq:epsilon_n_bound}
    c\abs{\abs{\sigma_n}-1}\leq \abs{\epsilon_n}\leq C\abs{\abs{\sigma_n}-1},\quad\forall n\in\bb{N}.  
 \end{equation}
\end{lemma}
\begin{proof}
First, note that by Remark \ref{rem:multiplicative_perturbations_nonzero_etc}, we may bound the constant in the $O(\abs{\abs{\sigma_n}-1}^3)$ term in \eqref{eq:I_of_abs_sigma} by some $C>0$ that may depend on $\alpha$, but not on $\abs{\sigma}$, i.e.
\begin{equation*}
 I(\abs{\sigma})=1+\sum^2_{m=1}a_m (\abs{\sigma}-1)^m+C(\abs{\sigma}-1)^3,\quad\forall \abs{\sigma}\in (1-\alpha,1+\alpha).
\end{equation*}
Thus by \eqref{eq:epsilon_n_mult_case}, the triangle inequality, and \eqref{eq:multiplicative_perturbations_nonzero_etc}, we obtain 
\begin{align*}
 \abs{\epsilon_n}&=\abs{a_2(\abs{\sigma_n}-1)+C(\abs{\sigma_n}-1)^2}\leq \max\{C,a_2\} \abs{\abs{\sigma_n}-1}\left(1+\abs{\abs{\sigma_n}-1}\right),
\end{align*}
which yields the upper bound in \eqref{eq:epsilon_n_bound} since $\abs{\abs{\sigma_n}-1}<\alpha$ by \eqref{eq:multiplicative_perturbations_nonzero_etc}. For the lower bound, we have for the same $C>0$ as above that
\begin{equation*}
  \abs{\epsilon_n}\geq \abs{a_2(\abs{\sigma_n}-1)}-\abs{C(\abs{\sigma_n}-1)^2}= \abs{\abs{\sigma_n}-1}\left(\abs{a_2}-C\abs{\abs{\sigma_n}-1}\right).
\end{equation*}
and for sufficiently small $\alpha$, we can make $\abs{a_2}-C\abs{\abs{\sigma_n}-1}$ strictly positive, by \eqref{eq:multiplicative_perturbations_nonzero_etc}.
\end{proof}
Again by making $\alpha$ small enough, we have from \eqref{eq:epsilon_n_mult_case} and \eqref{eq:a_1} that
\begin{equation}\label{eq:series_representation_log_I_n}
 \log I(\abs{\sigma_n})=\sum_{m\in\bb{N}}\frac{(-1)^{m-1}}{m}\left(I(\abs{\sigma_n})-1\right)^m=\sum_{m\in\bb{N}}\frac{(-1)^{m-1}}{m}\left(\epsilon_n-\frac{1}{2}\right)^m(\abs{\sigma_n}-1)^m.
\end{equation}

\begin{proposition}\label{prop:convergence_of_kakutanis_series_mult_case}
If the laws of $(U_n)_{n\in\bb{N}}$ and $(W_n)_{n\in\bb{N}}$ are equivalent, then $(\abs{\sigma_n}-1)_{n\in\bb{N}}$ is square-summable.
\end{proposition}
\begin{proof}
If the laws of $(U_n)_{n\in\bb{N}}$ and $(W_n)_{n\in\bb{N}}$ are equivalent, then by Theorem \ref{thm:kakutani_equivalence_product_measures}, \eqref{eq:kakutani_equivalence_product_measures_suffcon} holds. Without loss of generality, we may assume that $\abs{\sigma_n}\neq 1$ for all $n\in\bb{N}$, since otherwise the corresponding summand vanishes, by Lemma \ref{lem:necessary_condition_of_convergence_of_kakutanis_series}. Furthermore, (see Remark \ref{rem:multiplicative_perturbations_nonzero_etc}), we may also assume that \eqref{eq:multiplicative_perturbations_nonzero_etc} holds for some $0<\alpha<1$ that we can set as small as needed for \eqref{eq:series_representation_log_I_n} to hold and such that we can rewrite $\log\abs{\sigma_n}$ as a Mercator series in $\abs{\sigma_n}-1$, for all $n\in\bb{N}$. Using these conditions, \eqref{eq:integral_form_EVal_sqrt_RNdensity_mult_case}, \eqref{eq:I_of_abs_sigma}, and \eqref{eq:series_representation_log_I_n}, we obtain
\begin{equation}
 -\log \bb{E}[\sqrt{\varphi_n(U_n,\abs{\sigma_n})}]=-\sum_{m\in\bb{N}}\frac{(-1)^{m-1}}{m}\left[\frac{1}{2}+\left(\epsilon_n-\frac{1}{2}\right)^m\right](\abs{\sigma_n}-1)^m.
 \label{eq:taylor_series_of_kakutani_summands}
\end{equation}
The summand for $m=1$ is, by \eqref{eq:epsilon_n_mult_case} and \eqref{eq:I_of_abs_sigma},
\begin{equation*}
 -\epsilon_n(\abs{\sigma_n}-1)=(a_1(\abs{\sigma_n}-1)+1-I(\abs{\sigma_n}))=-a_2(\abs{\sigma_n}-1)^2+C(\abs{\sigma_n}-1)^3,
\end{equation*}
where $C>0$ may depend on $\alpha$ but not on $\abs{\sigma_n}$; see the proof of Lemma \ref{lem:lower_upper_bounds_for_abs_epsilon_n}.
Expanding the coefficient of $(\abs{\sigma_n}-1)^2$ for the summand in \eqref{eq:taylor_series_of_kakutani_summands} corresponding to $m=2$, we obtain that its value is
\begin{equation*}
 \frac{1}{2}\left[\frac{1}{2}+\left(\epsilon_n^2-\epsilon_n+\frac{1}{4}\right)\right]=\frac{3}{8}+\frac{1}{2}\left(\epsilon_n^2-\epsilon_n\right).
\end{equation*}
By Lemma \ref{lem:lower_upper_bounds_for_abs_epsilon_n}, $\epsilon_n^2-\epsilon_n$ converges to zero linearly in $\abs{\sigma_n}-1$. Therefore, combining this observation with the preceding two equations, we obtain from \eqref{eq:taylor_series_of_kakutani_summands} and \eqref{eq:a_2} that, for some $C'$ independent of $\abs{\sigma_n}$,
\begin{align*}
  -\log \bb{E}[\sqrt{\varphi_n(U_n,\abs{\sigma_n})}]&=\left(-\frac{5}{16}+\frac{3}{8}\right)(\abs{\sigma_n}-1)^2+C'((\abs{\sigma_n}-1)^3)
  \\
  &=\frac{1}{16}(\abs{\sigma_n}-1)^2+C'((\abs{\sigma_n}-1)^3).
\end{align*}
Therefore, by making $\alpha$ sufficiently small, we may ensure that for any $0<\delta<1/16$,
\begin{equation*}
 \frac{1}{(\abs{\sigma_n}-1)^2}\abs{-\log \bb{E}[\sqrt{\varphi_n(U_n,\abs{\sigma_n})}]-\frac{1}{16}(\abs{\sigma_n}-1)^2}=\abs{C'}\abs{\abs{\sigma_n}-1}<\delta.
\label{eq:abs_J_n_divby_tau_n_sq_lt_epsilon_prime}
\end{equation*}
Thus, for any $0<\delta<1/16$,
\begin{align*}
 0<\left(\frac{1}{16}-\delta\right)(\abs{\sigma_n}-1)^2 &<\frac{1}{16}(\abs{\sigma_n}-1)^2-\abs{-\log \bb{E}[\sqrt{\varphi_n(U_n,\abs{\sigma_n})}]-\frac{1}{16}(\abs{\sigma_n}-1)^2}.
\end{align*}
Since $-\abs{x}<x$ if $x>0$, and $-\abs{x}=x$ if $x<0$, it follows that the right-hand side of the inequality above is less than or equal to $-\log \bb{E}[\sqrt{\varphi_n(U_n,\abs{\sigma_n})}]$. Summing over $n$ yields
\begin{equation*}
 0<\left(\frac{1}{16}-\delta\right)\sum_{n\in\bb{N}}(\abs{\sigma_n}-1)^2<\sum_{n\in\bb{N}}-\log \bb{E}[\sqrt{\varphi_n(U_n,\abs{\sigma_n})}]<\infty,
\end{equation*}
and therefore the summability of Kakutani's series implies that $(\abs{\sigma_n}-1)_{n\in\bb{N}}$ is square-summable.
\end{proof}
\begin{proposition}\label{prop:tau_in_ell2_suffices_for_equivalence_mult_case}
If $(\abs{\sigma_n}-1)_{n\in\bb{N}}$ is square-summable, then the laws of $(U_n)_{n\in\bb{N}}$ and $(W_n)_{n\in\bb{N}}$ are equivalent.
\end{proposition}
\begin{proof}
 Without loss of generality, we may assume that the $\ell^2$ norm $\Vert (\abs{\sigma_n}-1)_{n\in\bb{N}}\Vert_2$ of  $(\abs{\sigma_n}-1)_{n\in\bb{N}}$ is strictly less than 1. Then, for $1<  p\leq q\leq +\infty$, the $\ell^p$-norm of $(\abs{\sigma_n}-1)_{n\in\bb{N}}$ is larger than or equal to the $\ell^q$-norm of $(\abs{\sigma_n}-1)_{n\in\bb{N}}$. Using this fact, the sum formula for a geometric progression, and the hypothesis on the $\ell^2$ norm of $(\abs{\sigma_n}-1)_{n\in\bb{N}}$, we have
\begin{equation}\label{eq:suffcon_for_interchange_summation_in_proof_of_sufficiency_of_ell_2_for_equivalence}
 \sum_{m\geq 2}\sum_{n\in\bb{N}}\abs{\abs{\sigma_n}-1}^m\leq \sum_{m\geq 2}\left(\sum_{n\in\bb{N}}\abs{\abs{\sigma_n}-1}^2\right)^{m/2} =\frac{\Vert(\abs{\sigma_n}-1)_{n\in\bb{N}}\Vert^2_2}{1-\Vert (\abs{\sigma_n}-1)_{n\in\bb{N}}\Vert_2^{1/2}}<\infty.
\end{equation}
Since $(\abs{\sigma_n}-1)_{n\in\bb{N}}$ converges to zero, we may choose $\alpha$ in \eqref{eq:multiplicative_perturbations_nonzero_etc} so small that, for $\epsilon_n$ defined in \eqref{eq:multiplicative_perturbations_nonzero_etc}, it holds that $\abs{\epsilon_n}<1/8$ for all $n\in\bb{N}$, by Lemma \ref{lem:lower_upper_bounds_for_abs_epsilon_n}. Since
\begin{equation*}\label{eq:abs_of_half_plus_mth_power_epsilon_minus_half_lt_1}
 \abs{\frac{1}{2}+\left(\epsilon_n-\frac{1}{2}\right)^m}
\le
 \begin{cases}
  \abs{\epsilon_n} & m=1
  \\
  \frac{1}{2}+\left(\abs{\epsilon_n}+\frac{1}{2}\right)^m & m\ge 2,
 \end{cases}
\end{equation*}
the condition that $\abs{\epsilon_n}<1/8$ implies that we may bound the left-hand side by 1 for all $m\in\bb{N}$. Using this observation and the triangle inequality, we may bound the right-hand side of  \eqref{eq:taylor_series_of_kakutani_summands} according to
\begin{equation*}
 -\sum_{m\in\bb{N}}\frac{(-1)^{m-1}}{m}\left[\frac{1}{2}+\left(\epsilon_n-\frac{1}{2}\right)^m\right](\abs{\sigma_n}-1)^m\leq \sum_{m\in\bb{N}}\abs{\abs{\sigma_n}-1}^m,
\end{equation*}
and hence from \eqref{eq:taylor_series_of_kakutani_summands} it follows that $-\log\bb{E}[\sqrt{\varphi_n(U_n,\abs{\sigma_n})}]\leq \Sigma_{m\in\bb{N}}\abs{\abs{\sigma_n}-1}^m$. Summing this inequality over $n$ yields
\begin{equation*}
  \sum_{n\in\bb{N}}-\log \bb{E}[\sqrt{\varphi_n(U_n,\abs{\sigma_n})}]\le  \sum_{n\in\bb{N}}\sum_{m\geq 2}\abs{\abs{\sigma_n}-1}^m,
\end{equation*}
where the right-hand side is finite, by changing the order of summation in \eqref{eq:suffcon_for_interchange_summation_in_proof_of_sufficiency_of_ell_2_for_equivalence}. Therefore, \eqref{eq:kakutani_equivalence_product_measures_suffcon} holds, and by Theorem \ref{thm:kakutani_equivalence_product_measures}, the laws of $(U_n)_{n\in\bb{N}}$ and $(W_n)_{n\in\bb{N}}$ are equivalent.
\end{proof}

\begin{remark}
 An anonymous referee pointed out that one can express the function $I$ defined in \eqref{eq:I_of_abs_sigma} in terms of complete elliptic integrals (see, e.g., \cite[Section~3.152, Formula~1]{gradshteyn_ryzhik} and \cite[Section~8.112, Formula 1]{gradshteyn_ryzhik}), and that the series representations of these elliptic integrals in \cite[Section~8.113]{gradshteyn_ryzhik} may provide an alternative method for obtaining the results that we presented above. 
\end{remark}

\section*{Acknowledgments}
The authors are supported by the Freie Universit\"{a}t Berlin within the Excellence Initiative of the German Research Foundation. The authors thank the editors and referees who read the article for their helpful comments.

\bibliographystyle{amsplain}
\bibliography{ecp_cm}

\end{document}